\documentclass[12pt]{amsart}
\usepackage{latexsym, amsthm, amscd, euscript, float, subfig}
\usepackage{color}
\usepackage{cite}
\usepackage{amssymb}
\usepackage{amsmath}
\usepackage{graphicx}
\usepackage[normalem]{ulem}
\usepackage{multirow}

\usepackage{array}

\newtheorem{theorem}{Theorem}[section]

\newtheorem{proposition}[theorem]{Proposition}

\newtheorem{remark}[theorem]{Remark}

\numberwithin{equation}{section}

\newcounter{minutes}
\divide\time by 60
\newcounter{hours}
\multiply\time by 60
\addtocounter{minutes}{-\time}
\newcommand{\IB}{\mathbb{B}}

\newcommand{\R}{\mathbb{R}}
\newcommand{\Rt}{{\mathbb R}^2}
\newcommand{\Rtbar}{\overline{{\mathbb R}}^2}
\newcommand{\B}{\mathbb{B}}
\newcommand{\Bn}{ {\mathbb{B}^n} }
\newcommand{\Rn}{ {\mathbb{R}^n} }

\newcommand{\beq}{\begin{equation}}
\newcommand{\eeq}{\end{equation}}

\DeclareMathOperator{\diam}{diam}
\begin{document}
\vspace*{-2cm}
\title[Mapping properties of the $\tilde{\tau}$-metric and the $u$-metric]
{Mapping properties of a scale invariant Cassinian metric and a Gromov hyperbolic metric}

%\date{\today}
% Time Stamp%%%%%%%%%%%%%%%%%
\def\thefootnote{}
\footnotetext{ \texttt{\tiny File:~\jobname .tex,
          printed: \number\day-\number\month-\number\year,
          \thehours.\ifnum\theminutes<10{0}\fi\theminutes}
} \makeatletter\def\thefootnote{\@arabic\c@footnote}\makeatother
%%%%%%%%%%%%%%%%%%%%%%%%%%%%%

\author[M. R. Mohapatra]{Manas Ranjan Mohapatra}
\address{Manas Ranjan Mohapatra, Discipline of Mathematics,
Indian Institute of Technology Indore,
Simrol, Khandwa Road, Indore 453 552, India
}
\email{mrm.iiti@gmail.com}

\author[S. K. Sahoo]{Swadesh Kumar Sahoo}
\address{Swadesh Kumar Sahoo, Discipline of Mathematics,
Indian Institute of Technology Indore,
Simrol, Khandwa Road, Indore 453 552, India}
\email{swadesh@iiti.ac.in}

\begin{abstract}
In this paper, we consider a scale invariant Cassinian metric and a Gromov hyperbolic metric.
We discuss a distortion property of the scale invariant Cassinian metric under M\"obius
maps of a punctured ball onto another punctured ball. 
We obtain a modulus of continuity 
of the identity map from a domain equipped with the scale invariant Cassinian metric 
(or the Gromov hyperbolic metric) onto the same domain equipped with the Euclidean metric.
The quasi-invariance properties of both the metrics under quasiconformal maps are  also established.
\\


\end{abstract}

\maketitle
\thispagestyle{empty}
\section{Introduction and Preliminaries}
An integral part of geometric function theory is to study the behavior of distances
under well-known classes of maps such as M\"obius maps, Lipschitz maps, quasiconformal 
maps, etc. There are some metrics which are M\"obius invariant and some are not. For example, the hyperbolic metric, the 
Apollonian metric \cite{Bea98,Has03,Ibr03} and the Seittenranta metric \cite{Sei99} 
are M\"obius invariant whereas the quasihyperbolic metric \cite{GO79,GP76} and 
the distance ratio metric \cite{Vuo85} are not. These metrics are also
known as the hyperbolic-type metrics in the literature.
Therefore, the study of quasi-invariance or distortion properties of the metrics that are 
not M\"obius invariant attracted many researchers in this field. The quasi-invariance properties 
of such metrics under quasiconformal maps are also of recent interest (see \cite{GO79,HKSV,KVZ14,Sei99}).
Note that the quasihyperbolic and the distance ratio
metrics do satisfy the bilipschitz property with the bilipschitz constant $2$ under M\"obius maps (see \cite[p.~36]{Vuo88}; also see\cite[Corollary~2.5]{GP76} and \cite[Proof of Theorem~4]{GO79}). 
Similar properties have also been studied recently for the Cassinian metric \cite{Ibr09} 
under M\"obius maps of the unit ball and of a punctured ball onto another punctured ball; 
see, for instance, \cite{IMSZ,KMS}.

Unless otherwise stated, we denote by $D$, a proper subdomain of $\Rn$. That is, we write $D\subsetneq \Rn$. 
A scale invariant Cassinian metric, recently introduced
by Ibragimov in \cite{Ibr16} and subsequently studied in \cite{IMS}, 
is defined by
$$\tilde{\tau}_D(x,y)=\log\left(1+\sup_{p\in \partial D}\frac{|x-y|}{\sqrt{|x-p||p-y|}}\right)\quad x,y\in D\subsetneq \Rn.
$$
The comparisons of the $\tilde{\tau}_D$-metric with the hyperbolic-type metrics and their metric ball inclusion 
properties are recently studied in the manuscript \cite{IMS}. The $\tilde\tau$-metric is important due to the following facts. Similar to the Apollonian metric, the (scale-invarint) Cassinian metric is described through ovals of Cassini \cite{Ibr09,Ibr16}.
The  hyperbolic-type metrics through their comparisons shares nice connection 
with the hyperbolic metric to characterize certain domains
such as quasidisks, uniform domains, John domains; see for instance ~\cite{GH00,GO79,KL98}. 
Charactrization of such domains in terms of the Cassinian metric or its scale invariant metric $\tilde\tau$
with other hyperbolic-type metrics are not known though a conjecture is stated in~\cite{Ibr09}.

Recall that the $\tilde{\tau}_D$-metric is M\"obius 
invariant in punctured spaces $\Rn\setminus \{p\}$, $p\in\Rn$ only \cite[Lemma~2.1]{Ibr16}.
Hence it is reasonable to study the quasi-invariance property of the $\tilde{\tau}_D$-metric under M\"obius maps of domains other than the punctured spaces $\Rn\setminus \{p\}$, $p\in\Rn$. In this regard,
we prove a distortion property of the $\tilde{\tau}_D$-metric under M\"obius maps of the punctured ball 
$\Bn\setminus\{0\}$ onto another punctured ball $\Bn\setminus\{a\}$, $0\neq a\in\Bn$. Note that a distortion property of the $\tilde{\tau}$-metric under M\"obius maps of the unit ball $\Bn:=\{x\in \Rn:|x|<1\}$ is recently established in \cite{Ibr16}. Hence, the quasi-invariance property of the $\tilde{\tau}_D$-metric under quasiconformal maps is also natural to study.

On the other hand, in 1987, Gromov introduced the notion of an abstract metric space; see \cite{Gro87}. 
One natural question was to investigate {\em whether a metric space is hyperbolic in the
sense of Gromov or not}? 
Ibragimov in \cite{Ibr11} introduced a metric, $u_Z$, which hyperbolizes the locally compact non-complete metric space $(Z,d)$ without changing its quasiconformal geometry, by 
$$u_Z(x,y)=2\log \frac{d(x,y)+\max\{{\rm dist(x, \partial Z)},{\rm dist(y, \partial Z)}\}}{\sqrt{{\rm dist(x, \partial Z)}\,{\rm dist(y, \partial Z)}}}, \quad x,y\in Z.
$$ 
For a domain $D\subsetneq \Rn$ equipped with the Euclidean metric, the {\em $u_D$-metric} is defined by
$$u_D(x,y)=2\log \frac{|x-y|+\max\{{\rm dist}(x,\partial D),{\rm dist}(y,\partial D)\}}
{\sqrt{{\rm dist}(x,\partial D)\,{\rm dist}(y,\partial D)}}, \quad x,y\in D.
$$
Recall that, as indicated in \cite{IMS}, the $u_D$-metric does not satisfy the domain monotonicity property and
it coincides with the distance ratio metric in punctured spaces $\Rn\setminus\{p\}$, for $p\in\Rn$. 
It is appropriate to recall here that Gromov hyperbolicity is preserved under rough quasi-isometries \cite[Theorem~3.17]{Vai05}; see also \cite{Has05}. This motivates us to study the $u_D$-metric in the
setting of hyperbolic-type metrics associated with the quasiconformal mappings.
Though the Gromov hyperbolic metric $u_D$ is compared with some of the hyperbolic-type metrics~\cite{IMS,Ibr11}, 
domain characterizations in terms of the $u_D$-metric are still open.

Secondly, the concept of uniform continuity is vastly used in metric spaces. The importance and applications of uniform continuity can be seen in many areas of 
mathematics and physics (see for instance, \cite{Kumaresan-book,Rudin-book}). 
A {\em modulus of continuity}\index{modulus of continuity} is a function $\omega:[0,\infty]\to [0,\infty]$ 
used to measure quantitatively the uniform continuity of functions.
Let $(X_j,d_j)$, $j=1,2$, be metric spaces. A function $f:\, X_1\to
X_2$ admits $\omega$ as {\em modulus of continuity} if and only if 
for all $x,y\in X_1$, $d_2(f(x),f(y))\le \omega(d_1(x,y))$. 
We also call such functions as {\em uniform continuous with modulus of continuity
$\omega$} (or {\em $\omega$-uniformly continuous}).
For instance, for $k>0$, the modulus $\omega(t)=kt$ describes the $k$-Lipschitz
functions, the moduli $\omega(t)=kt^{\alpha}$, $\alpha>0$, describe the H\"older
continuity, and so on. To simplify matters in this topic, we always assume
that $\omega(t)$ is an increasing homeomorphism.

Here we will mainly be motivated by geometric function theory and
therefore give a few related examples. 
Recall that the hyperbolic metric, $\rho_\Bn$, of the unit ball 
$\Bn$ is given by
$$
\rho_{\Bn}(x,y)=\inf_\gamma\int_\gamma \frac{2|{\rm d}z|}{1-|z|^2},
$$
where the infimum is taken over all rectifiable curves $\gamma\subset \IB^n$
joining $x$ and $y$.
If $X_1=\Bn=X_2$ and
$f:\,\Bn\to \Bn$ is quasiconformal, then the quasiconformal
counterpart of the Schwarz lemma says that $f:\, (\Bn,\rho_{\Bn})\to
(\Bn,\rho_{\Bn})$ is uniformly continuous.
If $X_1=\B^2$, $X_2=\Rt\setminus\{0,1\}$, the
Schottky theorem gives, in an explicit form, a growth estimate for
$|f(z)|$ in terms of $|z|$ when $f:\,\B^2\to \Rt\setminus\{0,1\}$
is an analytic function \cite[p. 685, 702]{Hay89}. In fact, Nevanlinna's
principle of the hyperbolic metric \cite[p. 683]{Hay89} yields an
estimate for the modulus of continuity of $f:\,(\B^2,\rho_{\B^2})\to
(X_2,d_2)$ where $d_2$ is the hyperbolic metric of the twice punctured
plane $X_2$. If $q$ is the chordal metric and
$f:\,(\B^2,\rho_{\B^2})\to (\Rtbar,q)$ is a meromorphic function,
then $f$ is normal (in the sense of Lehto and Virtanen \cite{LV73}) if and
only if it is uniformly continuous.  In the context of quasiregular
maps, uniform continuity has been discussed in \cite{Vuo85,Vu07}.
Uniform continuity of mappings with respect to the distance ratio metric and the quasihyperbolic metric in the unit ball has been discussed in \cite{KLV13}. 
In this connection, we consider the 
identity map $id\,:(\Bn,m_\Bn)\to (\Bn,|\cdot|)$ and prove that it is uniformly continuous, where $m_{\Bn}\in \{\tilde{\tau}_{\Bn},u_{\Bn}\}$. We also prove that the identity map $id\,:(D,\tilde{\tau}_D)\to (D,|\cdot|)$
is uniformly continuous, where $D\subsetneq \Rn$ is bounded.

Using the bilipschitz relation between the $u_D$-metric and the $\tilde{\tau}_D$-metric 
discussed in \cite[Theorem~3.5]{IMS}, 
we finally study the quasi-invariance property of these two metrics under quasiconformal maps of $\Rn$. 
Such problems are important in this context, as we already stated that
there are some metrics which are invariant under certain classes of mappings and some are not.  
If a metric is not invariant under certain classes of mappings, then it is reasonable to study 
its quasi-invariance property (also called the distortion property).
In particular, distortion properties attracted many researchers  
to prove several classical theorems in univalent function theory of the one and several 
complex variables; see \cite{Dur83,GK03}. Practically, distortion properties are also used
in converting a sphere to a flat surface when it is projected through mappings and the distortion constant
tells us how much it got stretched after it is projected. 
In general, four special properties (shape, area, distance,
and direction) subject to distortion are attracted many geometers. In this paper, one of our main objectives is 
to deal with distance properties subject to distortion. A surprising fact, in this context, is that though some metrics are
not invariant under quasiconformal mappings, domains that are characterized through such metric inequalities are invariant
under quasiconformal mappings of $\Rn$; see for instance~\cite{GO79}, that is, geometric properties or the definition of image of such domains remain unchanged.

\section{Distortion of the $\tilde{\tau}$-metric under M\"obius maps of a punctured ball}
Our objective in this section is to study the distortion property of the $\tilde{\tau}$- metric under M\"obius maps from a punctured ball onto another punctured ball.
Recall that distortion properties of the $\tilde{\tau}_{\Bn}$-metric under M\"obius maps of the unit ball $\Bn$ has recently been studied in \cite{Ibr16}.

\begin{theorem}\label{lip-bn}
Let $a\in \mathbb{B}^n$ and $f:\mathbb{B}^n\setminus\{0\} \to \mathbb{B}^n\setminus\{a\}$ 
be a M\"obius map with $f(0)=a$. Then for $x,y\in \mathbb{B}^n\setminus\{0\}$ we have
$$ \frac{1-|a|}{1+|a|} \tilde{\tau}_{\mathbb{B}^n\setminus \{0\}}(x,y) \le \tilde{\tau}_{\mathbb{B}^n\setminus\{a\}}(f(x),f(y))\le \frac{1+|a|}{1-|a|} \tilde{\tau}_{\mathbb{B}^n\setminus\{0\}}(x,y).
$$
\end{theorem}

\begin{proof}
If $a=0$, the proof is given in \cite{Ibr16}. 
Now, assume that $a\neq 0$. Let $\sigma$ be the inversion in the sphere $\mathbb S^{n-1}(a^\star,r)$, where 
$$
a^\star=\frac {a}{|a|^2}\qquad\text{and}\qquad r=\sqrt{|a^\star|^2-1}=\frac {\sqrt{1-|a|^2}}{|a|}.
$$
Note that the sphere $\mathbb S^{n-1}(a^\star,r)$ is orthogonal to $\mathbb{S}^{n-1}$ and that $\sigma(a)=0$. 
In particular, $\sigma$ is a M\"obius map with $\sigma(\IB^n\setminus\{a\})=\IB^n\setminus\{0\}$. 
Recall from \cite{Bea95} that 
\begin{equation}\label{inversionmap}
\sigma(x)=a^\star+\Big(\frac {r}{|x-a^\star|}\Big)^2\big(x-a^\star\big).
\end{equation}
Then $\sigma\circ f$ is an orthogonal matrix (see, for example, \cite[Theorem 3.5.1(i)]{Bea95}). 
In particular,  
\begin{equation}\label{mob1}
\Big|\sigma\big(f(x)\big)-\sigma\big(f(y)\big)\Big|=|x-y|.
\end{equation}
We will need the following property of $\sigma$ (see, for example, \cite[p. 26]{Bea95}):
\begin{equation}\label{mob2}
|\sigma(x)-\sigma(y)|=\frac {r^2|x-y|}{|x-a^\star||y-a^\star|}.
\end{equation}
It follows from (\ref{mob1}) and (\ref{mob2}) that
\begin{equation}\label{mob3}
|f(x)-f(y)|=\frac{|f(x)-a^\star||f(y)-a^\star|}{|a^\star|^2-1}|x-y|.
\end{equation}
Note that since $|f(z)|\le 1$ whenever $|z|\le 1$ and since $|a^\star|>1$, we have
$$ |a^\star|-1 \le |f(z)-a^\star|\le |a^\star|+1.
$$
Denote by $P=\min\{\sqrt{|f(x)-a||f(y)-a|}, \inf_{p\in \partial \Bn} \sqrt{|f(x)-p||f(y)-p|}\}$. 
Now, from the definition it follows that
$$\tilde{\tau}_{\Bn\setminus\{a\}}(f(x),f(y))=\log \left(1+\frac{|f(x)-f(y)|}{P}\right),
$$
and
$$\tilde{\tau}_{\Bn\setminus\{0\}} (x,y)=\log \left(1+\frac{|x-y|}{\min\{\sqrt{|x||y|},\inf_{z\in \partial \Bn}\sqrt{|x-z||y-z|}\}}\right).
$$
Here we have two choices for $P$.

\noindent{\bf Case I.}  $P=\sqrt{|f(x)-a||f(y)-a|}$

From (\ref{mob3}), it is clear that
$$|f(x)-a|=\frac{|f(x)-a^\star||a-a^\star|}{|a^\star|^2-1}|x|\quad\mbox{and } |f(y)-a|=\frac{|f(y)-a^\star||a-a^\star|}{|a^\star|^2-1}|y|.
$$
Now,
\begin{align*}
\tilde{\tau}_{\Bn\setminus\{a\}}(f(x),f(y)) &= \log \left(1+\frac{|f(x)-f(y)|}{\sqrt{|f(x)-a||f(y)-a|}}\right) \nonumber \\
&\le  \log \left(1+ \frac{1+|a|}{1-|a|} \frac{|x-y|}{\sqrt{|x||y|}} \right)\\
&\le  \frac{1+|a|}{1-|a|} \log \left(1+ \frac{|x-y|}{\sqrt{|x||y|}} \right)\le 
\frac{1+|a|}{1-|a|} \tilde{\tau}_{\Bn\setminus \{0\}}(x,y),
\end{align*}
where the second inequality follows from the well-known Bernoulli's inequality
\begin{equation}\label{Bernoulli-a>1}
\log(1+ax)\le a\log(1+x)\quad \mbox{ for } a\ge 1, x>0.
\end{equation}
Similarly, by taking the inverse mapping, we can prove that 
$$\frac{1-|a|}{1+|a|}\tilde{\tau}_{\Bn\setminus\{0\}}(x,y)\le \tilde{\tau}_{\Bn\setminus\{a\}}(f(x),f(y)).
$$

\noindent{\bf Case II.}  $P=\inf_{p\in \partial \Bn} \sqrt{|f(x)-p||f(y)-p|}$

This case follows from the proof of Theorem~7.1 in \cite{Ibr16}. This completes the proof of our theorem.
\end{proof}

\section{Uniform Continuity}

We begin this section with the following proposition which obtains the formula for the
$\tilde{\tau}_{\Bn}$-metric in the special cases when $x=ty$ for real $t\neq 0$. We observe that 
for $t>0$ the points $x$ and $y$ lie on a radial segment whereas for $t<0$ they are diametrically opposite. Assuming without of generality that $|x|\le |y|$
we have

\begin{proposition}\label{formula-bn}
Let $x,y\in \Bn$ with $x=ty,0\neq t\in \R$, and $|x|\le |y|$. Then we have the followings:
\begin{enumerate}
\item\label{formula-bn-1} If $t>0$, then 
$$\tilde{\tau}_{\Bn}(x,y)=\log\left(1+\frac{|x-y|}{\sqrt{(1-|x|)(1-|y|)}}\right), 
$$
\item\label{formula-bn-3} If $t<0$, then 
$$\tilde{\tau}_{\Bn}(x,y)=\log\left(1+\frac{|x-y|}{\sqrt{(1+|x|)(1-|y|)}}\right), 
$$
\end{enumerate}
\end{proposition}

\begin{proof}
The proof easily follows from the definition of the $\tilde\tau$-metric. Indeed, for the proof of (1), 
the maximal Cassinian oval touches the nearest boundary point $p$ of $\Bn$ on the direction of the radial segment. 
It follows that $|x-p||y-p|=(1-|x|)(1-|y|)$. Hence we obtain the desired formula.

Secondly, for the proof of (2), the maximal Cassinian oval touches the nearest boundary point $p$ close to $y$.
This yields $|x-p||y-p|=(1+|x|)(1-|y|)$. Now, the required formula follows from the definition of the $\tilde\tau$-metric.
\end{proof}

We now discuss the uniform continuity of the 
identity map $id:(D,m_D)\to (D,|.|)$, where $m_D$ is either the $u_D$-metric or the $\tilde{\tau}_D$-metric.
More precisely, first we consider a problem to find a bound, as sharp as possible, for the
modulus of continuity of the identity map
\begin{equation}\label{unif-cont2}
id:\, (D,\tilde{\tau}_{D})\to (D,|\cdot|\,),
\end{equation}
where $D\subsetneq \Rn$ is a bounded domain.
Secondly, we investigate such problem for the $u_D$-metric.

First, we obtain the modulus of continuity of the $id$ map \eqref{unif-cont2} when
$D=\Bn$.

\begin{theorem}\label{Jung-appl}
If $x,y\in \Bn$ are arbitrary and $w=|x-y|\,e_1/2$, then
$$\tilde{\tau}_{\Bn}(x,y)\ge \tilde{\tau}_{\Bn}(-w,w)=\log\left(1+\frac{2|x-y|}{\sqrt{4-|x-y|^2}}\right)
\ge c\,|x-y|\,,
$$
where $c\,(\approx 0.76)$ is the solution of the equation 
$$(4-t^2)(2t+\sqrt{4-t^2})\log\left(1+\frac{2t}{\sqrt{4-t^2}}\right)-8t=0.
$$
The first inequality becomes equality when $y=-x$.
\end{theorem}

\begin{proof}
Let $x,y\in \Bn$ with $|x|\le |y|$. Here we consider two cases.

\noindent{\bf Case I.} Suppose that $x$ and $y$ are lying on a diameter of $\Bn$ with $0\in [x,y]$.
Then it follows from Proposition~\ref{formula-bn}\eqref{formula-bn-3} that
$$\tilde{\tau}_{\Bn}(x,y)=\log\left(1+\frac{|x-y|}{\sqrt{(1+|x|)(1-|y|)}}\right)
$$
and hence
$$
\tilde{\tau}_{\Bn}(-w,w)=\log\left(1+\frac{2|w|}{\sqrt{1-|w|^2}}\right)=\log\left(1+\frac{2|x-y|}{\sqrt{4-|x-y|^2}}\right).
$$
With a suitable rotation and translation, geometrically, it can easily be seen that
the maximal Cassinian oval with foci at $x$ and $y$ will lie inside the maximal Cassinian oval with foci at $-w$ and $w$. Analytically, we say
$$\inf_{p\in \partial \Bn} \sqrt{|x-p||p-y|}\le \inf_{p\in \partial \Bn} \sqrt{|w+p||p-w|}.
$$
i.e.
$$\sqrt{(1+|x|)(1-|y|)}\le \sqrt{1-\frac{|x-y|^2}{4}},
$$
which is true.
Hence, $\tilde{\tau}_{\Bn}(x,y)\ge \tilde{\tau}_{\Bn}(-w,w)$.

If $x\in[0,y]$, then by Proposition~\ref{formula-bn}\eqref{formula-bn-1} we have
$$\tilde{\tau}_{\Bn}(x,y)=\log\left(1+\frac{|x-y|}{\sqrt{(1+|x|)(1-|y|)}}\right)
$$
and the proof follows similarly.

For the second inequality, we need to find the minimum of the function
$$\frac{1}{t} \log\left(1+\frac{2t}{\sqrt{4-t^2}}\right)\quad (t=|x-y|).
$$
By the derivative test, it can be seen that the minimum attains at the point $t\approx 1.16$ and the minimum value is approximately $0.76$.

\noindent{\bf Case II.} Let $x,y\in \Bn$ be arbitrary. Choose $y'\in \Bn$ such that $|x-y|=|x-y'|$ with $x$ and
$y'$ on a diameter of $\Bn$. With the same argument as we did in {\bf Case I}, we can show that
$$\tilde{\tau}_{\Bn}(x,y)\ge \tilde{\tau}_{\Bn}(x,y')\ge \tilde{\tau}_{\Bn}(-w,w).
$$
The proof is complete.
\end{proof}

Now, we obtain the modulus of continuity of the $id$ map \eqref{unif-cont2} when $D$ is a bounded proper subdomain of $\Rn$.

\begin{theorem}\label{gen-unif}
Let $D\subsetneq \Rn$ be a domain with
${\rm diam}\,D < \infty$ and $r=\sqrt{n/(2n+2)}\,{\diam}\, D$.
Then we have 
$$\tilde{\tau}_D(x,y)\ge \log\left(1+\frac{2|x-y|}{\sqrt{4r^2-|x-y|^2}}\right)\ge c\, \frac{|x-y|}{r}
$$
for all distinct $x,y\in D$ with equality in the first step
when $D=B^n(z,r)$ and $z=(x+y)/2$. Here $c$ is the same number as defined in {\rm Theorem}~$\ref{Jung-appl}$.
\end{theorem}

\begin{proof}
Given that $D\subsetneq \Rn$ is an arbitrary domain with ${\rm diam}\, D<\infty$. 
By the well-known Jung's theorem \cite[Theorem~11.5.8]{Ber87},
there exists $z\in \Rn$ with $D\subset B(z,r)$, where $r=\sqrt{n/(2n+1)}\,{\rm diam}\, D$. Then by the monotone property of $\tilde{\tau}_D$ we have
$$\tilde{\tau}_D(x,y)\ge \tilde{\tau}_{B(z,r)} (x,y).
$$
Without loss of generality assume that $z=0$. Choose $u,v\in B(0,r)$ in such a way that
$|u-v|=2|u|=|x-y|$. Then by Theorem~\ref{Jung-appl}, we have
$$\tilde{\tau}_D(x,y)\ge \tilde{\tau}_{B(z,r)} (x,y)\ge \tilde{\tau}_B(-u,u)=\log\left(1+\frac{2|x-y|}{\sqrt{4r^2-|x-y|^2}}\right).
$$
This completes the proof of our theorem.
\end{proof}

Now, we intend to obtain the modulus of continuity of the identity map:
$$ id:\, (\Bn,u_{\Bn})\to (\Bn,|\cdot|\,).
$$
\begin{theorem}\label{unif-u}
If $x,y\in \Bn$ are arbitrary and $w=|x-y|\,e_1/2$, then
$$u_{\Bn}(x,y)\ge u_{\Bn}(-w,w)=2\log\left(\frac{2+|x-y|}{2-|x-y|}\right)\ge |x-y|,
$$
where the first inequality becomes equality when $y=-x$.
\end{theorem}

\begin{proof}
Let $x,y\in \Bn$ with $|x|\le |y|$. We consider two cases.

\noindent{\bf Case I.} Suppose that $x$ and $y$ are lying on a diameter of $\Bn$. We have two possibilities. 
If $0\in [x,y]$, then it follows by the assumption that ${\rm dist}\,(x,\partial \Bn)=1-|x|\ge 1-|y|={\rm dist}\,(y,\partial \Bn)$. Then
$$u_{\Bn}(x,y)=2\log\left(\frac{|x-y|+1-|x|}{\sqrt{(1-|x|)(1-|y|)}}\right)$$
and hence
$$
u_{\Bn}(-w,w)=2\log\left(\frac{1+|w|}{1-|w|}\right)=2\log\left(\frac{2+|x-y|}{2-|x-y|}\right).
$$
By the AM-GM inequality we have
\begin{equation}\label{AM-GM}
\frac{1}{\sqrt{(1-|x|)(1-|y|)}}\ge \frac{2}{2-|x|-|y|}.
\end{equation}
To prove our claim, it is enough to show
$$\frac{|x-y|+1-|x|}{\sqrt{(1-|x|)(1-|y|)}}\ge \frac{2+|x-y|}{2-|x-y|}.
$$
Since $|x-y|=|x|+|y|$ and $|x|\le |y|$, we have
$$\frac{|x-y|+1-|x|}{\sqrt{(1-|x|)(1-|y|)}}\ge \frac{2(1+|y|)}{2-|x|-|y|}\ge \frac{2+|x|+|y|}{2-|x|-|y|}
=\frac{2+|x-y|}{2-|x-y|},
$$
where the first inequality follows from \eqref{AM-GM}.

If $x\in [0,y]$, we have ${\rm dist}\,(x,\partial \Bn)=1-|x|\ge 1-|y|={\rm dist}\,(y,\partial \Bn)$ and $|x-y|=|y|-|x|$.
Then by the definition of $u$-metric, we obtain
$$u_{\Bn}(x,y)=2\log\left(\frac{1+|y|-2|x|}{\sqrt{(1-|x|)(1-|y|)}}\right)
$$
and hence
$$
u_{\Bn}(-w,w)=2\log\left(\frac{1+|w|}{1-|w|}\right)=2\log\left(\frac{2+|y|-|x|}{2-|y|+|x|}\right).
$$
To show $u_{\Bn}(x,y)\ge u_{\Bn}(-w,w)$, it is enough to show that
$$\frac{1+|y|-2|x|}{\sqrt{(1-|x|)(1-|y|)}}\ge \frac{2+|y|-|x|}{2-|y|+|x|}.
$$
From \eqref{AM-GM}, we have
$$\frac{1+|y|-2|x|}{\sqrt{(1-|x|)(1-|y|)}}\ge \frac{2(1+|y|-2|x|)}{2-|x|-|y|}.
$$
Now our aim is to show
$$\frac{2(1+|y|-2|x|)}{2-|x|-|y|}\ge \frac{2+|y|-|x|}{2-|y|+|x|},
$$
or, equivalently,
$$2|y|-2|x|-|y|^2+6|x||y|-5|x|^2\ge 0.
$$
Since, $|x|\le |y|$, we have
\begin{eqnarray*}
2|y|-2|x|-|y|^2+6|x||y|-5|x|^2 &\ge & 2|y|-2|x|-|y|^2+|x|^2\\
&=& (1-|x|)^2-(1-|y|)^2 \ge 0.
\end{eqnarray*}

Again, since the function
$$f(t)=2\log\left(\frac{2+t}{2-t}\right)-t
$$
is increasing in $t$, the conclusion follows.

\noindent{\bf Case II.} Given that $x,y\in \Bn$ are arbitrary. Choose $y'\in \Bn$ such that $|x-y|=|x-y'|$ and $x,0,\mbox{ and }y'$ 
are co-linear. By geometry, it is clear that $|y'|\le |y|$. Hence,
$$u_{\Bn}(x,y)\ge u_{\Bn}(x,y')\ge u_{\Bn}(-w,w),
$$
where the first inequality follows from the definition and the second inequality follows from {\bf Case I}.
The proof is complete.
\end{proof}

\begin{remark}
It is remarkable that the domain monotonicity property of the $\tilde{\tau}_D$-metric 
plays a crucial role in the proof of 
Theorem~$\ref{gen-unif}$. Since the $u_D$-metric does not satisfy domain monotonicity property, 
it is not easy to extend Theorem~$\ref{unif-u}$ to arbitrary bounded domains of $\Rn$
in a similar manner.
\end{remark}

\section{Quasi-invariance property of the $u_D$-metric and the $\tilde{\tau}_D$-metric under quasiconformal maps}

Quasiconformal mappings are natural generalizations of conformal mappings. 
There are several equivalent definitions of quasiconformal mappings in the literature. Therefore,
it is appropriate to saying that, in this paper, we adopt the metric definition of the quasiconformality introduced by V\"ais\"al\"a in \cite{Vai71}.
We also refer to \cite{AVV,Vuo88} for it's definition and further developments in quasiconformal theory.
Let $D\subsetneq \Rn$ be a domain and let $f:D\to f(D)\subsetneq \Rn$ be a homeomorphism. The function
$f$ is said to be $K$-quasiconformal ($1\le K<\infty$), if the linear dilatation of $f$ at $x\in D$, defined by
\begin{equation}\label{qcd}
H(f,x)=\limsup_{r\to 0} \frac{\sup\{|f(x)-f(y)|:|x-y|=r\}}{\inf\{|f(x)-f(y)|:|x-y|=r\}},
\end{equation}
is bounded by $K$.
Moreover, if $f$ is $K$-quasiconformal then $\sup_{x\in D} H(f, x) \le c(n, K) < \infty$.

As an example of quasiconformal mapping, consider an $L$-bilipschitz map of $\mathbb{R}^n$, i.e., $f:\Rn \to \Rn$ satisfies
$$\frac{1}{L}|x-y|\le |f(x)-f(y)|\le L|x-y|, \quad x,y\in \Rn.
$$
Now, it is easy to verify from \eqref{qcd} that the $L$-bilipschitz map $f$ is $K$-quasiconformal with $K=L^2$. Further,
we notice that such an $L$-bilipschitz map $f$ is also $L^2$-bilipschitz with respect to the $\tilde{\tau}_D$-metric.
That is, 
$$\frac{1}{L^2}\tilde{\tau}_D(x,y)\le \tilde{\tau}_{f(D)}(f(x),f(y))\le L^2 \tilde{\tau}_D(x,y),\quad
x,y\in D.
$$
Indeed, 
\begin{eqnarray*}
\tilde{\tau}_{f(D)}(f(x),f(y)) &=& \log\left(1+\sup_{f(p)\in \partial f(D)}\frac{|f(x)-f(y)|}{\sqrt{|f(x)-f(p)||f(p)-f(y)|}}\right)\\
&\le & \log\left(1+\sup_{p\in \partial D}\frac{L^2|x-y|}{\sqrt{|x-p||p-y|}}\right)\\
&\le & L^2\,\log\left(1+\sup_{p\in \partial D}\frac{|x-y|}{\sqrt{|x-p||p-y|}}\right)=L^2\,\tilde{\tau}_D(x,y),
\end{eqnarray*}
where the first inequality follows from the bilipschitz condition on $f$ and the second inequality follows
from the well known Bernoulli's inequality \eqref{Bernoulli-a>1}. Observe that the distortion constant
$L^2$ is independent of the dimension of the space.
Now the question arises whether we can distort the $\tilde{\tau}_D$-metric if we replace the $L$-bilipschitz map by any arbitrary $K$-quasiconformal map?
The answer is ``yes" and the distortion constant will depend upon $n$ (the dimension of the space) and $K$. Briefly, our problem is the following:

{\em For $n\ge 1$ and $K\ge 1$, does there exist a constant $c$ depending only on $n$ and $K$ with the following property: 
if $f$ is a $K$-quasiconformal map of $(D,d)$ onto $(D',d')$, then
$$d'_{D'}(f(x),f(y))\le c \max\{d_D(x,y),d_D(x,y)^\alpha\}, \quad \alpha=K^{1/(1-n)},
$$
for all $x,y\in D$.
}

This problem is studied in different context  for different metrics. One of them is to see whether
$c\to 1$ as $K\to 1$. 
Obtaining a distortion constant (which tend to $1$ as $K\to 1$) under a $K$-quasiconformal map of $\Rn$ for hyperbolic-type metrics in general is a challenging problem in geometric function theory. Hence, the existance of distortion constant $c$ without knowing its limit is also investigated by researchers.
In this regard, Gehring and Osgood first proved the quasi-invariance property of the quasihyperbolic metric under quasiconformal maps of $\Rn$ in \cite[Theorem~3]{GO79}.
Here, the distortion constant doesn't tend to $1$ as $K\to 1$. 
However, the distortion constant for the Seittenranta metric~\cite[Theorem~1.2]{Sei99}, and hence for the hyperbolic metric \cite[Corollary~2.10]{KVZ14} approaches $1$ as $K\to 1$. Hence we are interested to study the quasi-invariance property of the $\tilde{\tau}_D$-metric followed by the quasi-invariance property of the $u_{D}$-metric under  quasiconformal maps of $\Rn$. In the proof we need the quasi-invariance property of 
the distance ratio metric. 
Recall that the distance ratio metric, $\tilde{j}_D$, is defined by
$$\tilde{j}_D(x,y)=\log\left(1+\frac{|x-y|}{\min\{{\rm dist}(x,\partial D),{\rm dist}(y,\partial D)\}}\right),\quad x,y\in D.
$$

\begin{theorem}\label{qinv-tau}
For $n\ge 1$ and $K\ge 1$,
if $f:\Rn\to \Rn$ is a $K$-quasiconformal mapping which maps $D\subsetneq \Rn$ onto $D'\subsetneq \Rn$ then there exists a constant
$C_1$ depending only on $n$ and $K$ such that
$$\tilde{\tau}_{D'}(f(x),f(y))\le C_1\max\{\tilde{\tau}_D(x,y),\tilde{\tau}_D(x,y)^\alpha\}
$$
for all $x,y\in D$, where $\alpha=K^{1/(1-n)}$.
\end{theorem}
\begin{proof}
For all $x,y\in D$, we have
\begin{eqnarray*}
\tilde{\tau}_{D'}(f(x),f(y)) \le  \tilde{j}_D(f(x),f(y)) &\le & C\max\{\tilde{j}_D(x,y),\tilde{j}_D(x,y)^\alpha\}\\
&\le &C \max\{2\tilde{\tau}_D(x,y),2^\alpha \tilde{\tau}_D(x,y)^\alpha\}\\
&\le &C_1 \max\{\tilde{\tau}_D(x,y),\tilde{\tau}_D(x,y)^\alpha\},
\end{eqnarray*}
where the first and third inequality follows from \cite[Theorem~4.2, 4.3]{Ibr16}, the second inequality follows from 
\cite[Lemma~2.3]{HKSV}, and the constant $C_1$ is depending upon $n$ and $K$. This completes the proof.
\end{proof}

The following quasi-invariance property holds true for the $u_D$-metric under quasiconformal mappings of $\Rn$.

\begin{theorem}
For $n\ge 1$ and $K\ge 1$,
if $f:\Rn\to \Rn$ is a $K$-quasiconformal mapping which maps $D\subsetneq \Rn$ onto $D'\subsetneq \Rn$ then there exists a constant
$C_2$ depending only on $n$ and $K$ such that
$$u_{D'}(f(x),f(y))\le C_2 \max\{u_D(x,y),u_D(x,y)^\alpha\}
$$
for all $x,y\in D$, where $\alpha=K^{1/(1-n)}$.
\end{theorem}

\begin{proof}
Recall that the $u_D$-metric and the $\tilde{\tau}_D$-metric are bilipschitz equivalent.
Indeed,
$$2 \tilde{\tau}_D(x,y) \le u_D(x,y)\le 4\tilde{\tau}_D(x,y).
$$
(see, \cite[Theorem~4.5]{IMS}). 
The proof now follows from Theorem~\ref{qinv-tau}.
\end{proof}

\bigskip
\noindent
{\bf Acknowledgement.} The authors wish to thank the referees for their valuable comments.

\end{document}